\documentclass[11pt]{amsart}

\theoremstyle{plain}
\newtheorem{theorem}                {Theorem}      [section]
\newtheorem*{theorem*}                {Theorem \ref{main theorem}}
\newtheorem{proposition}  [theorem]  {Proposition}
\newtheorem{corollary}    [theorem]  {Corollary}

\theoremstyle{definition}

\newtheorem{remark}       [theorem]  {Remark}
\newtheorem{definition}   [theorem]  {Definition}

\DeclareMathOperator{\trace}{trace} 
\DeclareMathOperator{\grad}{grad}

 \DeclareMathOperator{\id}{I}

\DeclareMathOperator{\Span}{span}
 
\DeclareMathOperator{\cst}{constant}
 
 \DeclareMathOperator{\im}{Im}

\numberwithin{equation}{section}

\begin{document}

\title[Biharmonic surfaces with parallel mean curvature]{Biharmonic surfaces with parallel mean curvature in complex space forms}

\author{Dorel Fetcu}
\author{Ana Lucia Pinheiro}

\address{Department of Mathematics and Informatics\\
Gh. Asachi Technical University of Iasi\\
Bd. Carol I, 11 \\
700506 Iasi, Romania} \email{dfetcu@math.tuiasi.ro}

\curraddr{Department of Mathematics, Federal University of Bahia, Av.
Ademar de Barros s/n, 40170-110 Salvador, BA, Brazil}

\address{Department of Mathematics, Federal University of Bahia, Av.
Ademar de Barros s/n, 40170-110 Salvador, BA, Brazil} \email{anapinhe@ufba.br}

\thanks{The first
author was partially supported by a grant of the Romanian National Authority
for Scientific Research, CNCS -- UEFISCDI, project number
PN-II-RU-TE-2011-3-0108.}

\begin{abstract} We classify complete biharmonic surfaces with parallel mean curvature vector field and non-negative Gaussian curvature in complex space forms.
\end{abstract}



\maketitle

\section{Introduction}

The \textit{biharmonic maps} were suggested in 1964 by J. Eells and J. H. Sampson, as a generalization of \textit{harmonic maps} (see 
\cite{JEJS}). Thus, whereas a harmonic map $\varphi:(M,g)\rightarrow(N,h)$ between two Riemannian manifolds
is a critical point of the \textit{energy functional}
$$
E(\varphi)=\frac{1}{2}\int_{M}|d\varphi|^{2}\ v_{g},
$$
a biharmonic map is a critical point of the
\textit{bienergy functional}
$$
E_{2}(\varphi)=\frac{1}{2}\int_{M}|\tau(\varphi)|^{2}\ v_{g},
$$
where $\tau(\varphi)=\trace\nabla d\varphi$ is the tension field that
vanishes for harmonic maps. The Euler-Lagrange equation corresponding to the
bienergy functional was obtained by G. Y. Jiang in 1986 (see \cite{GYJ}):
\begin{align*}
\tau_{2}(\varphi)&=\Delta\tau(\varphi)-\trace\bar R(d\varphi,\tau(\varphi))d\varphi\\
&=0
\end{align*}
where $\tau_{2}(\varphi)$ is the \textit{bitension field} of $\varphi$, $\Delta=\trace(\nabla^{\varphi})^2 =\trace(\nabla^{\varphi}\nabla^{\varphi}-\nabla^{\varphi}_{\nabla})$ is the rough Laplacian defined on
sections of $\varphi^{-1}(TN)$ and
$\bar R$ is the curvature tensor of $N$, given by $\bar R(X,Y)Z=[\bar\nabla_X,\bar\nabla_Y]Z-\bar\nabla_{[X,Y]}Z$. Since any harmonic map
is biharmonic, we are interested in non-harmonic biharmonic maps,
which are called \textit{proper-biharmonic}.

A \textit{biharmonic submanifold} in a Riemannian manifold is a
submanifold for which the inclusion map is biharmonic. In Euclidean
space the biharmonic submanifolds are the same as those defined by
Chen in \cite{BYC}, characterized by 
$\Delta H=0$, where $H$ is the mean curvature vector field and
$\Delta$ is the rough Laplacian.

Many classification results for proper-biharmonic submanifolds in space forms, i.e., spaces with constant sectional curvature, were obtained in the last decade (see, for example, \cite{BMO1,BMO-AM,BO}) and the next step was to look for such submanifolds in spaces with non-constant sectional curvature. Some very important examples of such ambient spaces are the complex space forms, i.e., simply-connected K\"ahler manifolds with constant holomorphic sectional curvature. Recent papers as \cite{FLMO,FO-TMJ,U,S-GJM,Z} treated the subject of proper-biharmonic submanifolds in complex space forms and several classification results and examples were found.

On the other hand, submanifolds with parallel mean curvature vector (pmc submanifolds) or with constant mean curvature (cmc submanifolds) in Riemannian manifolds proved to be very good candidates for providing nice examples of proper-biharmonic submanifolds (see, for example, \cite{BMO1,BMO-AM,BO,FOR,OW,Z}).

In our paper, we first consider pmc surfaces in complex space forms and prove a Simons type formula for the Laplacian of the squared norm of the holomorphic differential $Q^{(2,0)}$, defined on such a surface, introduced in \cite{DF}. Then we use this formula to show that, if $\Sigma^2$ is a complete pmc surface with non-negative Gaussian curvature, then the surface is flat or $Q^{(2,0)}$ vanishes on $\Sigma^2$. Next, we investigate the biharmonicity of these surfaces and, using a reduction of codimension theorem of J.~H.~Eschenburg and R. Tribuzy in
\cite{ET} and the above mentioned result, we obtain the following classification theorem.

\begin{theorem*} Let $\Sigma^2$ be a complete proper-biharmonic pmc surface with non-negative Gaussian curvature in $\mathbb{C}P^n(\rho)$. Then $\Sigma^2$ is totally real and either
\begin{enumerate}

\item $\Sigma^2$ is pseudo-umbilical and its mean curvature is equal to $\sqrt{\rho}/2$. Moreover, 
$$
\Sigma^2=\pi(\widetilde\Sigma^2)\subset\mathbb{C}P^n(\rho),\quad n\geq 3,
$$
where $\pi:\mathbb{S}^{2n+1}(\rho/4)\rightarrow\mathbb{C}P^n(\rho)$ is the Hopf fibration and the horizontal lift $\widetilde\Sigma^2$ of $\Sigma^2$ is a complete minimal surface in a small hypersphere $\mathbb{S}^{2n}(\rho/2)\subset\mathbb{S}^{2n+1}(\rho/4)$; or

\item $\Sigma^2$ lies in $\mathbb{C}P^2(\rho)$ as a complete Lagrangian proper-biharmonic pmc surface. Moreover, if $\rho=4$, then
$$
\Sigma^2=\pi\Big(\mathbb{S}^1\Big(\sqrt{\frac{9\pm\sqrt{41}}{20}}\Big)\times\mathbb{S}^1\Big(\sqrt{\frac{11\mp\sqrt{41}}{40}}\Big)\times\mathbb{S}^1\Big(\sqrt{\frac{11\mp\sqrt{41}}{40}}\Big)\Big)\subset\mathbb{C}P^2(4),
$$
where $\pi:\mathbb{S}^{5}(1)\rightarrow\mathbb{C}P^2(4)$ is the Hopf fibration; or

\item $\Sigma^2$ lies in $\mathbb{C}P^3(\rho)$ and
$$
\Sigma^2=\gamma_1\times\gamma_2\subset\mathbb{C}P^3(\rho),
$$
where $\gamma_1:\mathbb{R}\rightarrow\mathbb{C}P^2(\rho)\subset \mathbb{C}P^3(\rho)$ is a holomorphic helix of order $4$ with curvatures
$$
\kappa_1=\sqrt{\frac{7\rho}{6}},\quad \kappa_2=\frac{1}{2}\sqrt{\frac{5\rho}{42}},\quad \kappa_3=\frac{3}{2}\sqrt{\frac{\rho}{42}},
$$
and complex torsions
$$
\tau_{12}=-\tau_{34}=\frac{11\sqrt{14}}{42},\quad \tau_{23}=-\tau_{14}=\frac{\sqrt{70}}{42},\quad \tau_{13}=\tau_{24}=0,
$$
and $\gamma_2:\mathbb{R}\rightarrow\mathbb{C}P^3(\rho)$ is a holomorphic circle with curvature $\kappa=\sqrt{\rho/2}$ and complex torsion $\tau_{12}=0$. Moreover, the curves $\gamma_1$ and $\gamma_2$ always exist and are unique up to holomorphic isometries.
\end{enumerate}
\end{theorem*}

\noindent {\bf Conventions.} We work in the $C^{\infty}$ category, that means manifolds, metrics, connections,
and maps are smooth. The Lie algebra of vector fields on a surface $\Sigma^2$ is denoted by $C(T\Sigma^2)$. 
The surfaces are always assumed to be connected and orientable.

\noindent {\bf Acknowledgments.} The authors wish to thank Cezar Oniciuc for many useful comments and helpful discussions. The first author would like to thank the Department of Mathematics of the Federal University of Bahia in Salvador for providing a very stimulative work environment during the preparation of this paper.

\section{Preliminaries}

Let $N^n(\rho)$ be a complex space form with complex dimension $n$, complex
structure $(J,\langle,\rangle)$, and constant holomorphic
sectional curvature $\rho$, i.e., $N^n(\rho)$ is either $\mathbb{C}P^n(\rho)$,
or $\mathbb{C}^n$, or $\mathbb{C}H^n(\rho)$, as $\rho>0$,
$\rho=0$, and $\rho<0$, respectively. Then the curvature tensor of $(N(\rho),J,\langle,\rangle)$ is given by
\begin{align}\label{curv}
\bar R(X,Y)Z=&\frac{\rho}{4}\{\langle Y,Z\rangle X-\langle X,Z\rangle Y+\langle JY,Z\rangle JX-\langle JX,Z\rangle JY\\\nonumber &+2\langle
X,JY\rangle JZ\}.
\end{align}

Let $\Sigma^2$ be a surface immersed in $N^{n}(\rho)$. The second fundamental form $\sigma$ of $\Sigma^2$ is defined by the equation of Gauss
$$
\bar\nabla_XY=\nabla_XY+\sigma(X,Y),
$$
while the shape operator $A$ and the normal connection $\nabla^{\perp}$ are given by the equation of Weingarten
$$
\bar\nabla_XV=-A_VX+\nabla^{\perp}_XV,
$$
for any vector fields $X$ and $Y$ tangent to the surface and any vector field $V$ normal to $\Sigma^2$, where $\bar\nabla$ and $\nabla$ are the Levi-Civita connections of $N^{n}(\rho)$  and $\Sigma^2$, respectively.

\begin{definition} If the mean curvature vector $H$ of the surface $\Sigma^2$ is
parallel in the normal bundle, i.e.
$\nabla^{\perp}H=0$, then $\Sigma^2$ is called a \textit{pmc surface}.
\end{definition}

We end this section by recalling some notions and results from the theory of Frenet curves in complex space forms, which we shall use later.

Let $\gamma:I\subset\mathbb{R}\rightarrow N^n(\rho)$ be a curve parametrized by
arc-length. Then $\gamma$ is called a {\it Frenet curve of
osculating order} $r$, $1\leq r\leq 2n$, if there exist $r$
orthonormal vector fields $\{X_1=\gamma',\ldots,X_r\}$ along
$\gamma$ such that
$$
\bar\nabla_{X_{1}}X_{1}=\kappa_{1}X_{2},\quad
\bar\nabla_{X_{1}}X_{i}=-\kappa_{i-1}X_{i-1} + \kappa_{i}X_{i+1},\quad\ldots\quad,
\bar\nabla_{X_{1}}X_{r}=-\kappa_{r-1}X_{r-1},
$$
for all $i\in\{2,\ldots,r-1\}$, where $\{\kappa_{1},\kappa_{2},\ldots,\kappa_{r-1}\}$ are positive
functions on $I$ called the {\it curvatures} of $\gamma$. These equations are called the {\it Frenet equations} of $\gamma$.

A Frenet curve of osculating order $r$ is called a {\it helix of
order $r$} if $\kappa_i=\cst>0$ for $1\leq i\leq r-1$. A helix
of order $2$ is called a {\it circle}, and a helix of order $3$ is
simply called {\it helix}.

S.~Maeda and Y.~Ohnita defined in \cite{SMYO} the {\it
complex torsions} $\tau_{ij}$ of the curve $\gamma$ by $\tau_{ij}=\langle
X_i, JX_j\rangle$, where $1\leq i<j\leq r$. A helix of order $r$
is called a {\it holomorphic helix of order $r$} if all its complex
torsions are constant. We note that a circle is always a holomorphic circle.

In \cite{MA} there are proved the following existence results.

\begin{theorem}[\cite{MA}]\label{MA1} For
given positive constants $\kappa_1$, $\kappa_2$, and $\kappa_3$, there exist four
equivalence classes of holomorphic helices of order $4$ in $\mathbb{C}P^2(\rho)$
with curvatures $\kappa_1$, $\kappa_2$, and $\kappa_3$ with respect to
holomorphic isometries of $\mathbb{C}P^2(\rho)$. The four classes are defined by
certain relations on the complex torsions and they are: when
$\kappa_1\neq\kappa_3 $
$$
\begin{array}{|c|lll|}
\hline
I_1& \tau_{12}=\tau_{34}=\mu&\tau_{23}=\tau_{14}=\kappa_2\mu/(\kappa_1+\kappa_3)&\tau_{13}=\tau_{24}=0\\
\hline
I_2& \tau_{12}=\tau_{34}=-\mu&\tau_{23}=\tau_{14}=-\kappa_2\mu/(\kappa_1+\kappa_3)&\tau_{13}=\tau_{24}=0\\
\hline
I_3& \tau_{12}=-\tau_{34}=\nu&\tau_{23}=-\tau_{14}=\kappa_2\nu/(\kappa_1-\kappa_3)&\tau_{13}=\tau_{24}=0\\
\hline
I_4& \tau_{12}=-\tau_{34}=-\nu&\tau_{23}=-\tau_{14}=-\kappa_2\nu/(\kappa_1-\kappa_3)&\tau_{13}=\tau_{24}=0\\
\hline
\end{array}
$$
where
$$
\mu=\dfrac{\kappa_1+\kappa_3}{\sqrt{\kappa_2^2+(\kappa_1+\kappa_3)^2}}\quad\textnormal{and}\quad
\nu=\dfrac{\kappa_1-\kappa_3}{\sqrt{\kappa_2^2+(\kappa_1-\kappa_3)^2}},
$$
and when $\kappa_1=\kappa_3$ the classes $I_3$ and $I_4$ are substituted
by
$$
\begin{array}{|c|ll|}
\hline
I_3'& \tau_{12}=\tau_{34}=\tau_{13}=\tau_{24}=0 & \tau_{23}=-\tau_{14}=1\\
\hline
I_4'& \tau_{12}=\tau_{34}=\tau_{13}=\tau_{24}=0 & \tau_{23}=-\tau_{14}=-1\\
\hline
\end{array}
$$
\end{theorem}

\begin{theorem}[\cite{MA}]\label{MA2} For any positive number $\kappa$ and for any number $\tau$, such that $|\tau|<1$, there exits a holomorphic circle with curvature $\kappa$ and complex torsion $\tau$ in any complex space form.
\end{theorem}

\section{A Simons type formula for pmc surfaces in complex space forms}

Let $(N^n(\rho),J,\langle,\rangle)$ be a complex space form, with constant holomorphic sectional curvature $\rho$ and complex dimension $n$, and $\Sigma^2$ be a pmc surface in $N^n(\rho)$.

In \cite{DF} it is proved that the $(2,0)$-part $Q^{(2,0)}$ of the quadratic form $Q$ defined on $\Sigma^2$ by
\begin{equation}\label{eq:Q}
Q(X,Y)=8|H|^2\langle A_HX,Y\rangle+3\rho\langle X,T\rangle\langle Y,T\rangle,
\end{equation}
where $T$ is the tangent part $(JH)^{\top}$ of $JH$, is holomorphic.

Using this holomorphic differential we shall prove the following result.

\begin{theorem}\label{thm} Let $\Sigma^2$ be a complete non-minimal pmc surface with
non-negative Gaussian curvature $K$ isometrically immersed in a complex
space form $N^n(\rho)$, $\rho\neq 0$. Then one of the following holds:
\begin{enumerate}
\item the surface is flat;

\item there exists a point $p\in\Sigma^2$ such that $K(p)>0$ and $Q^{(2,0)}$ vanishes on $\Sigma^2$.
\end{enumerate}
\end{theorem}

\begin{proof}
First, we recall a Simons type equation obtained by S.-Y. Cheng and S.-T. Yau (equation $2.8$ in \cite{CY}), which generalizes J. Simons' result in \cite{JS}. Let $M$ be an $m$-dimensional Riemannian manifold, and consider a symmetric operator $S$ on $M$, that satisfies the Codazzi equation $(\nabla_XS)Y=(\nabla_YS)X$, where $\nabla$ is the Levi-Civita connection on the manifold. Then, we have
\begin{equation}\label{delta}
\frac{1}{2}\Delta|S|^2=|\nabla S|^2+\sum_{i=1}^{m}\lambda_i(\trace S)_{ii}+\frac{1}{2}\sum_{i,j=1}^{m}R_{ijij}(\lambda_i-\lambda_j)^2,
\end{equation}
where $\lambda_i$, $1\leq i\leq m$, are the eigenvalues of $S$, and $R_{ijkl}$ are the components of the Riemannian curvature of $M$.

Next, let us consider the following operator $S$, defined on our surface $\Sigma^2$ by
\begin{equation}\label{eq:S}
S=8|H|^2A_H+3\rho\langle T,\cdot\rangle T-\Big(\frac{3\rho}{2}|T|^2+8|H|^4\Big)\id.
\end{equation}
We shall prove that $|S|^2$ is a bounded subharmonic function on the surface.

First, it is easy to see that
\begin{equation}\label{eq:SQ}
\langle SX,Y\rangle=Q(X,Y)-\frac{\trace Q}{2}\langle X,Y\rangle,
\end{equation}
which implies that $S$ is symmetric and traceless. It is also easy to see that $Q^{(2,0)}$ vanishes on $\Sigma^2$ if and
only if $S=0$ on the surface.

Using \eqref{eq:SQ}, since $Q^{(2,0)}$ is holomorphic, just as in \cite[Proposition 3.3]{B} one can prove that $S$ satisfies the Codazzi equation $(\nabla_XS)Y=(\nabla_YS)X$, where $\nabla$ is the Levi-Civita connection on the surface.

Then, from equation \eqref{delta} and the fact that $\trace S=0$, we easily get
\begin{equation}\label{eq:Simons}
\frac{1}{2}\Delta|S|^2=2K|S|^2+|\nabla S|^2,
\end{equation}
where $K$ is the Gaussian curvature of the surface.

Now, let us consider the local orthonormal frame field $\{E_3=H/|H|,E_4,\ldots,E_{2n}\}$ in the normal bundle, and denote $A_{\alpha}=A_{E_{\alpha}}$. It follows that $\trace A_3=2|H|$ and $\trace A_{\alpha}=0$, for all $\alpha>3$.

From the definition \eqref{eq:S} of $S$, we have, after a straightforward computation,
$$
\det A_3=\frac{1}{|H|^2}\det A_H=|H|^2-\frac{1}{128|H|^6}|S|^2-\frac{9\rho^2}{256|H|^6}|T|^4+\frac{3\rho}{64|H|^6}\langle ST,T\rangle,
$$
and then, by using the equation of Gauss of $\Sigma^2$ in $N$,
\begin{align}\label{Gauss}
R(X,Y)Z=&\frac{\rho}{4}\{\langle Y, Z\rangle X-\langle X, Z\rangle Y+\langle JY,Z\rangle JX-\langle JX, Z\rangle JY\\\nonumber &+2\langle X,JY\rangle JZ\}+\sum_{\alpha=3}^{2n}\{\langle A_{\alpha}Y,Z\rangle A_{\alpha}X-\langle A_{\alpha}X,Z\rangle A_{\alpha}Y\},
\end{align}
the Gaussian curvature can be written as
\begin{align}\label{K}
K=&\frac{\rho}{4}(1+3\cos^2\theta)+|H|^2-\frac{1}{128|H|^6}|S|^2-\frac{9\rho^2}{256|H|^6}|T|^4+\frac{3\rho}{64|H|^6}\langle ST,T\rangle\\\nonumber
&+\sum_{\alpha>3}\det A_{\alpha},
\end{align}
where $\theta=\langle JE_1,E_2\rangle$ is the K\"ahler angle function of $\Sigma^2$, $\{E_1,E_2\}$ being a local orthonormal positively oriented frame field in the tangent bundle.

Since $\trace A_{\alpha}=0$, it follows that $\det A_{\alpha}\leq 0$, for all $\alpha>3$. Therefore, as $K\geq 0$, we get the following global formula
$$
-\frac{1}{128|H|^2}|S|^2+\frac{3\rho}{64|H|^6}\langle ST,T\rangle-\frac{9\rho^2}{256|H|^6}|T|^4+|H|^2+\frac{\rho}{4}(1+3\cos^2\theta)\geq 0.
$$
From $|\langle ST,T\rangle|\leq\frac{1}{\sqrt{2}}|T||S|$, since $|T|\leq|JH|=|H|$, we have $\rho\langle ST,T\rangle\leq\frac{|\rho|}{\sqrt{2}}|H|^2|S|$, which implies
$$
-\frac{1}{128|H|^6}|S|^2+\frac{3|\rho|}{64\sqrt{2}|H|^4}|S|+|H|^2+\frac{\rho}{4}(1+3\cos^2\theta)\geq 0.
$$

In the following we shall prove that $|S|$ is bounded. We have two cases as $\rho<0$ or $\rho>0$.

If $\rho<0$ we have
$$
-\frac{1}{128|H|^6}|S|^2+\frac{3\rho}{64\sqrt{2}|H|^4}|S|+|H|^2\geq 0
$$
and then $|S|\leq \frac{(\sqrt{9\rho^2+256|H|^4}-3\rho)|H|^2}{\sqrt{2}}$.

When $\rho>0$ we get
$$
-\frac{1}{128|H|^6}|S|^2+\frac{3\rho}{64\sqrt{2}|H|^4}|S|+|H|^2+\rho\geq 0,
$$
which is equivalent to $|S|\leq \frac{(\sqrt{9\rho^2+256\rho|H|^2+256|H|^4}+3\rho)|H|^2}{\sqrt{2}}$.

Since the surface is complete and has non-negative Gaussian curvature, we see, using a result of A. Huber in \cite{H}, that $\Sigma^2$ is a parabolic space. From the above calculation and \eqref{eq:Simons}, we get that $|S|^2$ is a bounded subharmonic function, which implies that $|S|$ is a constant. Then, from \eqref{eq:Simons}, we get that $K=0$ on $\Sigma^2$ or there exists a point $p\in\Sigma^2$ such that $K(p)>0$ and then $S=0$ on the surface, which, as we have seen, is equivalent to $Q^{(2,0)}=0$.
\end{proof}

\begin{remark}\label{r_thm} For a surface $\Sigma^2$ as in Theorem \ref{thm} we have $|S|=\cst$ and $\nabla S=0$.
\end{remark}

\section{Biharmonic pmc surfaces in $\mathbb{C}P^n(\rho)$}

In order to prove our main result we shall need the following theorem.

\begin{theorem}[\cite{BMO-AM}]\label{thm split}
A submanifold $\Sigma^m$ in a Riemannian manifold $N$, with second fundamental form $\sigma$,
mean curvature vector field $H$, and shape operator $A$, is biharmonic if and only if
$$
\begin{cases}
-\Delta^{\perp}H+\trace\sigma(\cdot,A_H\cdot)+\trace(\bar R(\cdot,H)\cdot)^{\perp}=0\\
\frac{m}{2}\grad|H|^2+2\trace A_{\nabla^{\perp}_{\cdot}H}(\cdot)+2\trace(\bar R(\cdot,H)\cdot)^{\top}=0,
\end{cases}
$$
where $\Delta^{\perp}$ is the Laplacian in the normal bundle and $\bar R$ is the curvature tensor of $N$.
\end{theorem}

Using the formula \eqref{curv} of the curvature tensor of a complex space form $N(\rho)$, we get the following result.

\begin{corollary}\label{coro split}
Let $\Sigma^2$ be a pmc surface in a complex space form $(N(\rho),J,\langle,\rangle)$. Then $\Sigma^2$ is biharmonic if and only if
\begin{equation}\label{eq:bih}
\begin{cases}
\trace\sigma(\cdot,A_H\cdot)=\frac{\rho}{4}\{2H-3(JT)^{\perp}\}\\
(JT)^{\top}=0,
\end{cases}
\end{equation}
where $T$ is the tangent part of $JH$ and $(JT)^{\perp}$ and $(JT)^{\top}$ are the normal and the tangent part of $JT$, respectively.
\end{corollary}

\begin{remark} It is easy to see, from the first equation of \eqref{eq:bih}, that for a proper-biharmonic pmc surface we have 
$$
0<|A_H|^2=\frac{\rho}{4}\{2|H|^2+3|T|^2\}
$$
which implies that $\rho>0$, and, therefore, such surfaces exist only in $\mathbb{C}P^n(\rho)$.
\end{remark}

\begin{proposition}\label{p_cst}
If $\Sigma^2$ is a proper-biharmonic pmc surface in $\mathbb{C}P^n(\rho)$ then $T$ has constant length.
\end{proposition}

\begin{proof} The map
$p\in\Sigma^2\rightarrow(A_H-\mu\id)(p)$, where $\mu$ is a constant, is analytic, and,
therefore, either $\Sigma^2$ is a pseudo-umbilical surface (at every
point), or $H$ is an umbilical direction on a closed set without interior
points (see \cite{AdCT,DF}). We shall denote by $W$ the set of points where $H$ is not an
umbilical direction. Since in the second case this set is open and dense in
$\Sigma^2$, when the surface is not pseudo-umbilical we shall work on $W$ and then extend our results throughout $\Sigma^2$ by continuity.

If $\Sigma^2$ is pseudo-umbilical then $JH$ is normal to the surface, i.e., $T=0$ on the surface (see \cite{Sato1}).

Let us assume now that $\Sigma^2$ is not pseudo-umbilical and let $N$ be the normal part of $JH$. Then, for any vector field $X$ tangent to the surface, we have
\begin{align*}
\bar\nabla_XJH&=-J\bar\nabla_XH=-JA_HX\\&=\nabla_XT+\sigma(X,T)-A_NX+\nabla^{\perp}_XN
\end{align*}
and, therefore,
\begin{equation}\label{eq:t1}
\langle\nabla_XT,T\rangle=\langle A_NX,T\rangle+\langle A_HX,JT\rangle=\langle A_NX,T\rangle,
\end{equation}
since, from the second equation of \eqref{eq:bih}, we know that $JT$ is normal.

It easy to see that
$$
\langle N,H\rangle=0\quad\textnormal{and}\quad\langle N,JT\rangle=0
$$
and, again using \eqref{eq:bih}, that
$$
\langle N,JX\rangle=0,\quad\forall X\in C(T\Sigma^2).
$$
Then, from the first equation of \eqref{eq:bih}, we get that
\begin{equation}\label{eq:t2}
\trace(A_HA_N)=0.
\end{equation}
Moreover, using the Ricci equation
\begin{equation}\label{Ricci}
\langle
R^{\perp}(X,Y)H,V\rangle=\langle[A_H,A_V]X,Y\rangle+\langle
\bar R(X,Y)H,V\rangle,\quad\forall V\in C(N\Sigma^2),
\end{equation}
we obtain 
\begin{equation}\label{eq:t3}
[A_H,A_N]T=0,
\end{equation}
since $R^{\perp}(X,Y)H=0$ and $\langle\bar R(X,T)H,N\rangle=0$, for tangent vector fields $X$ and $Y$.

Next, consider a point $p\in W$ and an orthonormal basis $\{e_1,e_2\}$ in $T_p\Sigma^2$ such that $A_He_i=\lambda_ie_i$, $i\in\{1,2\}$. Obviously, we have $\lambda_1\neq\lambda_2$ and we can write $A_H$ and $A_N$ with respect to $\{e_1,e_2\}$ as
$$
A_H=\left(\begin{array}{cc}\lambda_1&0\\0&\lambda_2\end{array}\right)\quad\textnormal{and}\quad A_N=\left(\begin{array}{cc}a&b\\b&-a\end{array}\right),
$$
since $N\perp H$, i.e., $\trace A_N=0$. From \eqref{eq:t2} we get $a=0$ and then \eqref{eq:t3} becomes
$$
(\lambda_2-\lambda_1)b(\langle T,e_2\rangle e_1-\langle T,e_1\rangle e_2)=0.
$$
Therefore, at $p$, we have that either $T=0$ or $b=0$. We can see that in both cases $\langle A_NX,T\rangle=0$, which implies that equation \eqref{eq:t1} reduces to
$$
X(|T|^2)=2\langle\nabla_XT,T\rangle=0
$$ 
for any tangent vector $X$. It follows that $X(|T|^2)=0$ for any tangent vector field $X$ on $\Sigma^2$, which means that $|T|$ is constant on the surface.
\end{proof}

\begin{remark}\label{r_T} If $|T|=\cst\neq 0$ we have $\nabla_XT=A_NX=0$ for any tangent vector field $X$. Indeed, if $T\neq 0$ everywhere, since $JT$ is a normal vector field, it follows that $\Sigma^2$ is a totally real surface. Then we get 
$$
\bar\nabla_XJH=-J\bar\nabla_XH=-JA_HX\in C(N\Sigma^2)
$$ 
which means that $\nabla_XT=A_NX$. On the other hand, we have $\langle\bar R(X,Y)H,N\rangle=0$ for any tangent vector fields $X$ and $Y$, and then, from the Ricci equation \eqref{Ricci}, one sees that $[A_H,A_N]=0$. Using this equation and \eqref{eq:t2} in the same way as in the proof of Proposition \ref{p_cst}, and since $T\neq 0$ implies that $H$ is not umbilical on an open dense set, we obtain $A_N=0$ on this set and, therefore, on the whole surface. 
\end{remark}

\begin{proposition}\label{p_umb}
If $\Sigma^2$ is a complete proper-biharmonic pmc surface in $\mathbb{C}P^n(\rho)$ with non-negative Gaussian curvature $K$ and $T=0$, then $n\geq 3$ and $\Sigma^2$ is pseudo-umbilical and totally real. Moreover, the mean curvature of $\Sigma^2$ is $|H|=\sqrt{\rho}/2$.
\end{proposition}

\begin{proof} From Corollary \ref{coro split} we see that the pmc surface $\Sigma^2$ with $T=0$ is proper-biharmonic if and only if  
\begin{equation}\label{eq:p_umb}
\trace\sigma(\cdot,A_H\cdot)=\frac{\rho}{2}H.
\end{equation}

Now, from Theorem \ref{thm}, we know that either the Gaussian curvature $K$ vanishes identically on the surface, or there exists a point $p\in\Sigma^2$ such that $K(p)>0$ and $Q^{(2,0)}=0$ on $\Sigma^2$.

In the second case, since $T=0$ and $Q^{(2,0)}=0$, it is easy to see that $\Sigma^2$ is pseudo-umbilical and then totally real (see \cite{Sato1}). From \eqref{eq:p_umb}, we get that $|A_H|^2=(\rho/2)|H|^2$, but since $\Sigma^2$ is pseudo-umbilical, we also have $|A_H|^2=2|H|^4$, which means that $|H|=\sqrt{\rho}/2$.

If the surface is flat, we shall prove first that it is also totally real. Since $JH$ is a normal vector field to $\Sigma^2$, we have
\begin{align*}
\bar\nabla_XJH&=J\bar\nabla_XH=-JA_HX\\ 
&=-A_{JH}X+\nabla^{\perp}_XJH.
\end{align*}

Let us now consider an orthonormal basis $\{e_1,e_2\}$ in $T_p\Sigma^2$, where $p\in\Sigma^2$, such that $A_He_i=\lambda_ie_i$, $i\in\{1,2\}$. It follows that $JA_He_i=\lambda_iJe_i$ and, for $i\neq j$, we have
$$
\langle A_{JH}e_i,e_j\rangle=\langle JA_He_i,Ee_j\rangle=\lambda_i\langle Je_i,e_j\rangle.
$$
Thus, we obtained $\lambda_1\langle Je_1,e_2\rangle=\lambda_2\langle Je_2,e_1\rangle$, which means that
$$
0=(\lambda_1+\lambda_2)\langle Je_1,e_2\rangle=2|H|^2\langle Je_1,e_2\rangle.
$$
Therefore, we have $\langle Je_1,e_2\rangle=0$, i.e., $\Sigma^2$ is totally real.

In the following, we will prove that $\Sigma^2$ is also pseudo-umbilical. Assume that it is not so and we will work on the set $W$ defined in the proof of Proposition \ref{p_cst}. Let $p$ be a point in $W$, consider a basis $\{e_1,e_2\}$ in $T_p\Sigma^2$ such that $A_He_i=\lambda_ie_i$, and extend the $e_i$ to vector fields $E_i$ in a neighborhood of $p$. First, using the expression \eqref{curv} of the curvature tensor of $\mathbb{C}P^n(\rho)$, we obtain,
$\langle\bar R(E_2,E_1)H,JE_1\rangle=0$ and then, from the Ricci equation \eqref{Ricci}, $\langle [A_H,A_{JE_1}]E_1,E_2\rangle=0$, which can be written at $p$ as
$$
(\lambda_2-\lambda_1)\langle A_{JE_1}E_1,E_2\rangle=0.
$$
In the same way, we can also show that $(\lambda_1-\lambda_2)\langle A_{JE_2}E_2,E_1\rangle=0$. 

First, since $\lambda_1\neq\lambda_2$, we get that
$\langle A_{JE_1}E_1,E_2\rangle=\langle A_{JE_2}E_2,E_1\rangle=0$. Using the fact that $\Sigma^2$ is totally real, it is easy to verify that
$$
\langle\sigma(X,Y),JZ\rangle=\langle\sigma(X,Z),JY\rangle,\quad\forall X,Y,Z\in C(T\Sigma^2),
$$
and then, at $p$, we obtain
$$
\langle A_{JE_2}E_1,E_1\rangle=\langle\sigma(E_1,E_1),JE_2\rangle=\langle\sigma(E_1,E_2),JE_1\rangle=\langle A_{JE_1}E_1,E_2\rangle=0
$$
and
$$
\langle A_{JE_1}E_2,E_2\rangle=\langle A_{JE_2}E_2,E_1\rangle=0.
$$
Since $JH$ is normal to $\Sigma^2$ is equivalent to $\trace A_{JE_1}=\trace A_{JE_2}=0$, we have just proved that $A_{JE_1}=A_{JE_2}=0$ at $p$.

Next, for any normal vector field $U$ which is also orthogonal to $H$, $JE_1$, and $JE_2$, we have $\langle\bar R(X,Y)H,U\rangle=0$ and then, from the Ricci equation \eqref{Ricci}, $[A_H,A_U]=0$. Since $H$ is not umbilical on $W$, this implies that, with respect to $\{E_1,E_2\}$, we have, at $p$,
$$
A_H=\left(\begin{array}{cc}a+|H|^2&0\\0&-a+|H|^2\end{array}\right)\quad\textnormal{and}\quad A_U=\left(\begin{array}{cc}b&0\\0&-b\end{array}\right),
$$
with $a\neq 0$. From \eqref{eq:p_umb} we have that $\trace(A_HA_U)=0$, which implies, using the above expressions, that $A_U=0$.

Now, we consider a local orthonormal frame field in the normal bundle of $\Sigma^2$, as follows $\{E_3=H/|H|,E_4=JE_1,E_5=JE_2,E_6,\ldots,E_{2n}\}$ and,
since the surface is flat, from the Gauss equation \eqref{Gauss} of $\Sigma^2$ in $\mathbb{C}P^n(\rho)$, at $p$ we get
$$
0=K=\frac{\rho}{4}+\sum_{\alpha=3}^{2n}\det A_{\alpha}=\frac{\rho}{4}+\det A_3=\frac{\rho}{4}+|H|^2-\frac{a^2}{|H|^2}.
$$
From \eqref{eq:p_umb} we have $|A_H|^2=2a^2+2|H|^4=(\rho/2)|H|^2$ and, therefore, $K=2|H|^2$ at $p$, which means that $|H|=0$. This is a contradiction, since $\Sigma^2$ is proper-biharmonic. Hence, the surface is pseudo-umbilical in this case too.

Finally, we have that, for any vector field $X$ tangent to the surface, the vector field $JX$ is normal and orthogonal to both $H$ and $JH$, which are also normal vector fields. Therefore, one obtains $n\geq 3$ and we conclude.
\end{proof}

\begin{proposition}\label{t_neq_0}
If $\Sigma^2$ is a complete proper-biharmonic pmc surface in $\mathbb{C}P^n(\rho)$ with non-negative Gaussian curvature $K$ and $T\neq 0$, then the surface is flat and $\nabla A_H=0$.
\end{proposition}

\begin{proof} Since $|T|=\cst\neq 0$ on $\Sigma^2$, from the second equation of \eqref{eq:bih}, we know that our surface is totally real. In \cite{CO} it is proved that the $(2,0)$-part $\widetilde Q^{(2,0)}$ of the quadratic form
$$
\widetilde Q(X,Y)=\langle A_HX,Y\rangle,
$$
defined on a pmc totally real surface, is holomorphic. Consider the traceless part $\phi_H=A_H-|H|^2\id$ of $A_H$. Since $\widetilde Q^{(2,0)}$ is holomorphic, working in the same way as in \cite[Proposition 3.3]{B}, we can prove that $\phi_H$ satisfies the Codazzi equation $(\nabla_XS)Y=(\nabla_YS)X$. Hence, from equation \eqref{delta}, we have
$$
\frac{1}{2}\Delta|\phi_H|^2=2K|\phi_H|^2+|\nabla\phi_H|^2.
$$

Let us assume now that there exists a point $p\in\Sigma^2$ such that $K(p)>0$. Then, from Theorem \ref{thm}, we have that $S=0$, which implies that
$$
|\phi_H|^2=|A_H|^2-2|H|^4=\frac{9\rho^2}{128|H|^4}|T|^4=\cst\neq 0,
$$
which means that $K=0$ on $\Sigma^2$ and this is a contradiction. 

Hence the surface is flat. Since $\Sigma^2$ is proper-biharmonic, it follows, from the first equation of \eqref{eq:bih}, that $|\phi_H|^2$ is bounded. Thus $|\phi_H|^2$ is a bounded subharmonic function on a parabolic space and, therefore, a constant, which implies $\nabla A_H=\nabla\phi_H=0$.
\end{proof}

\begin{remark} In the proof of Proposition \ref{t_neq_0} we used the fact that $\widetilde Q^{(2,0)}$ is holomorphic when $\widetilde Q$ is defined on a totally real pmc surface in a complex space form. In \cite{LO} it is proved that, if $\Sigma^2$ is a proper-biharmonic surface with constant mean curvature in a Riemannian manifold, then $\widetilde Q^{(2,0)}$ is holomorphic.
\end{remark}

Before proving our main result, let us briefly recall a property of the Hopf fibration (see \cite{R}). Let $\pi:\mathbb{C}^{n+1}\setminus\{0\}\rightarrow\mathbb{C}P^n(\rho)$ be the natural projection and $\mathbb{S}^{2n+1}(\rho/4)=\{z\in\mathbb{C}^{n+1}:\langle z,z\rangle=4/\rho\}$. The restriction of $\pi$ to the sphere $\mathbb{S}^{2n+1}(\rho/4)\subset\mathbb{C}^{n+1}$ is the Hopf fibration $\pi:\mathbb{S}^{2n+1}(\rho/4)\rightarrow\mathbb{C}P^n(\rho)$ and it is a Riemannian submersion. Now, let $i:\Sigma^m\rightarrow \mathbb{C}P^n(\rho)$ be a totally real isometric immersion. Then this immersion can be lifted locally (or globally, if $\Sigma^m$ is simply connected) to a horizontal immersion $\widetilde i:\widetilde\Sigma^{m}\rightarrow\mathbb{S}^{2n+1}(\rho/4)$. Conversely, if $\widetilde i:\widetilde\Sigma^{m}\rightarrow\mathbb{S}^{2n+1}(\rho/4)$ is a horizontal isometric immersion, then $\pi(\widetilde i):\Sigma^m\rightarrow\mathbb{C}P^n(\rho)$ is a totally real isometric immersion. Moreover, we have $\pi_{\ast}\widetilde\sigma=\sigma$, where $\widetilde\sigma$ and $\sigma$ are the second fundamental forms of the immersions $\widetilde i$ and $i$, respectively.

We shall also use the following theorem.

\begin{theorem}[\cite{BMO1}]\label{thm:teza} Let $\Sigma^m$ be a proper-biharmonic cmc submanifold in
$\mathbb{S}^n(\rho/4)$ with mean curvature vector field $H$. Then $|H|\in(0,\sqrt{\rho}/2]$ and, moreover, $|H|=\sqrt{\rho}/2$ if and only if $\Sigma^m$ is
minimal in a small hypersphere
$\mathbb{S}^{n-1}(\rho/2)\subset\mathbb{S}^n(\rho/4)$.
\end{theorem}

We are ready now to prove our main result.

\begin{theorem}\label{main theorem} Let $\Sigma^2$ be a complete proper-biharmonic pmc surface with non-negative Gaussian curvature in $\mathbb{C}P^n(\rho)$. Then $\Sigma^2$ is totally real and either
\begin{enumerate}

\item $\Sigma^2$ is pseudo-umbilical and its mean curvature is equal to $\sqrt{\rho}/2$. Moreover, 
$$
\Sigma^2=\pi(\widetilde\Sigma^2)\subset\mathbb{C}P^n(\rho),\quad n\geq 3,
$$
where $\pi:\mathbb{S}^{2n+1}(\rho/4)\rightarrow\mathbb{C}P^n(\rho)$ is the Hopf fibration and the horizontal lift $\widetilde\Sigma^2$ of $\Sigma^2$ is a complete minimal surface in a small hypersphere $\mathbb{S}^{2n}(\rho/2)\subset\mathbb{S}^{2n+1}(\rho/4)$; or

\item $\Sigma^2$ lies in $\mathbb{C}P^2(\rho)$ as a complete Lagrangian proper-biharmonic pmc surface. Moreover, if $\rho=4$, then
$$
\Sigma^2=\pi\Big(\mathbb{S}^1\Big(\sqrt{\frac{9\pm\sqrt{41}}{20}}\Big)\times\mathbb{S}^1\Big(\sqrt{\frac{11\mp\sqrt{41}}{40}}\Big)\times\mathbb{S}^1\Big(\sqrt{\frac{11\mp\sqrt{41}}{40}}\Big)\Big)\subset\mathbb{C}P^2(4),
$$
where $\pi:\mathbb{S}^{5}(1)\rightarrow\mathbb{C}P^2(4)$ is the Hopf fibration; or

\item $\Sigma^2$ lies in $\mathbb{C}P^3(\rho)$ and
$$
\Sigma^2=\gamma_1\times\gamma_2\subset\mathbb{C}P^3(\rho),
$$
where $\gamma_1:\mathbb{R}\rightarrow\mathbb{C}P^2(\rho)\subset \mathbb{C}P^3(\rho)$ is a holomorphic helix of order $4$ with curvatures
$$
\kappa_1=\sqrt{\frac{7\rho}{6}},\quad \kappa_2=\frac{1}{2}\sqrt{\frac{5\rho}{42}},\quad \kappa_3=\frac{3}{2}\sqrt{\frac{\rho}{42}},
$$
and complex torsions
$$
\tau_{12}=-\tau_{34}=\frac{11\sqrt{14}}{42},\quad \tau_{23}=-\tau_{14}=\frac{\sqrt{70}}{42},\quad \tau_{13}=\tau_{24}=0,
$$
and $\gamma_2:\mathbb{R}\rightarrow\mathbb{C}P^3(\rho)$ is a holomorphic circle with curvature $\kappa=\sqrt{\rho/2}$ and complex torsion $\tau_{12}=0$. Moreover, the curves $\gamma_1$ and $\gamma_2$ always exist and are unique up to holomorphic isometries.
\end{enumerate}
\end{theorem}

\begin{proof} Let $\Sigma^2$ be a complete proper-biharmonic pmc surface with non-negative Gaussian curvature $K$ and mean curvature vector field $H$ in $\mathbb{C}P^n(\rho)$. Let $T$ and $N$ be the tangent and the normal parts of $JH$, respectively. As we have seen in Proposition \ref{p_cst}, the length of $T$ is constant along the surface. We shall consider two cases, as $T=0$ or $T\neq 0$ on $\Sigma^2$.

{\bf Case I: $T=0$.} From Proposition \ref{p_umb} we know that $n\geq 3$ and the surface is pseudo-umbilical and totally real with mean curvature $|H|=\sqrt{\rho}/2$. Consider the Hopf fibration $\pi:\mathbb{S}^{2n+1}(\rho/4)\rightarrow\mathbb{C}P^n(\rho)$ and the horizontal lift $\widetilde\Sigma^2$ of $\Sigma^2$ in $\mathbb{S}^{n+1}(\rho/4)$. Then, from \cite[Theorem 1]{R}, we have that $\widetilde\Sigma^2$ is pseudo-umbilical in $\mathbb{S}^{2n+1}(\rho/4)$ and has parallel mean curvature vector field. Moreover, its mean curvature is constant and equal to $\sqrt{\rho}/2$. Next, using the relation between the bitension fields of the immersions $i:\Sigma^2\rightarrow \mathbb{C}P^n(\rho)$ and $\widetilde i:\widetilde\Sigma^2\rightarrow\mathbb{S}^{2n+1}(\rho/4)$, given in \cite[Theorem 3.3]{FLMO}, together with $(JH)^{\top}=T=0$, we get that $\Sigma^2$ is proper-biharmonic if and only if $\widetilde\Sigma^2$ is proper-biharmonic. We apply Theorem \ref{thm:teza} to conclude that $\widetilde\Sigma^2$ is a complete minimal surface in a small hypersphere $\mathbb{S}^{2n}(\rho/2)\subset\mathbb{S}^{2n+1}(\rho/4)$.

{\bf Case II: $T\neq 0$.} In this case $\Sigma^2$ is totally real and, from Proposition \ref{t_neq_0}, flat. From the same Proposition \ref{t_neq_0} we also know that $\nabla A_H=0$, which means that the eigenfunctions of $A_H$ are actually constants. Since $\Sigma^2$ is pseudo-umbilical implies that $T=0$, it follows that the surface does not have umbilical points.

Now, let $U$ be a normal vector field orthogonal to $H$ and to $J(T\Sigma^2)$. Then, it is easy to see that $\langle\bar R(X,Y)H,U\rangle=0$ and, from the Ricci equation \eqref{Ricci}, we get that $[A_H,A_U]=0$. Since $H$ is not umbilical, this implies that $A_H$ and $A_U$ can be simultaneously diagonalized. 

On the other hand, the first equation of \eqref{eq:bih} shows that $\trace(A_HA_U)=0$ and, therefore, that $A_U=0$.

Let us consider the global orthonormal frame field $\{E_1=T/|T|,E_2\}$ on the surface. We know, from Remark~\ref{r_T}, that $\nabla E_1=0$ and then $\nabla E_2=0$. 

Next, if $|T|=|H|$, i.e., if $JH$ is a tangent vector field, we consider the subbundle $L=\Span\{JE_1,JE_2\}$ of the normal bundle. Since $JH$ is tangent, we get that $H\in L$ and, therefore, for any normal vector field $U\perp L$, we have $A_U=0$, which means that $\im\sigma\subset L$. It is also easy to see that $\dim(T\Sigma^2\oplus L)=4$ and $J(T\Sigma^2\oplus L)=T\Sigma^2\oplus L$, which implies that $\bar R(X,Y)Z\in L$, for all vector fields $X,Y,Z\in T\Sigma^2\oplus L$, i.e., $T\Sigma^2\oplus L$ is invariant by $\bar R$. In the following, we shall prove that $L$ is parallel, i.e., if $U$ is a normal vector field orthogonal to $L$, then $U$ is also orthogonal to $\nabla^{\perp}L$. Indeed, for all tangent vector fields $X$ and $Y$, we obtain
$$
\langle\nabla_X^{\perp}JY,U\rangle=\langle\bar\nabla_XJY,U\rangle=\langle J\bar\nabla_XY,U\rangle=\langle J\nabla_XY+J\sigma(X,Y),U\rangle=0,
$$
since $\im\sigma\subset L$ and $J(T\Sigma^2\oplus L)=T\Sigma^2\oplus L$. Therefore, since $\bar\nabla\bar R=0$, we can use \cite[Theorem~2]{ET} to show that there exists a $4$-dimensional totally geodesic submanifold of $\mathbb{C}P^n(\rho)$ such that $\Sigma^2$ lies in this submanifold. Since $J(T\Sigma^2\oplus L)=T\Sigma^2\oplus L$, we get that $\Sigma^2$ is a complete Lagrangian proper-biharmonic pmc surface in $\mathbb{C}P^2(\rho)$. When $\rho=4$, these surfaces were determined in \cite{S-GJM} as follows
$$
\Sigma^2=\pi\Big(\mathbb{S}^1\Big(\sqrt{\frac{9\pm\sqrt{41}}{20}}\Big)\times\mathbb{S}^1\Big(\sqrt{\frac{11\mp\sqrt{41}}{40}}\Big)\times\mathbb{S}^1\Big(\sqrt{\frac{11\mp\sqrt{41}}{40}}\Big)\Big)\subset\mathbb{C}P^2(4).
$$

Now, assume that $|T|<|H|$. Since $\Sigma^2$ is totally real, we can consider the following local normal orthonormal frame field
$$
\Big\{E_3=JE_1,E_4=JE_2,E_5=\frac{1}{|N|}JN,E_6=\frac{1}{|N|}N,E_7,\ldots,E_{2n}\Big\},
$$
where $E_3$, $E_4$, $E_5$, and $E_6$ are globally defined. It can be easily verified that $H$ is orthogonal to $E_4$, $E_6$, and $E_{\alpha}$, where $\alpha\in\{7,\ldots,2n\}$, and, therefore, that
\begin{equation}\label{eq:H}
H=-|T|E_3-|N|E_5.
\end{equation}

All vector fields $E_{\alpha}$, $\alpha\geq 7$, are orthogonal to $H$ and to $J(T\Sigma^2)$, which means that $A_{\alpha}=0$, $\alpha\geq 7$ and, therefore, $\im\sigma\subset L=\Span\{E_3,E_4,E_5,E_6\}$. Moreover, the bundle $T\Sigma^2\oplus L$ is invariant by $J$ and by $\bar R$. Let $U$ be a normal vector field, orthogonal to $L$. Using the facts that $\nabla T=0$, $\im\sigma\subset L=\Span\{E_3,E_4,E_5,E_6\}$, and $J(T\Sigma^2\oplus L)=T\Sigma^2\oplus L$, one obtains
$$
\langle\nabla_X^{\perp}JY,U\rangle=\langle\bar\nabla_XJY,U\rangle=\langle J\bar\nabla_XY,U\rangle=\langle J\nabla_XY+J\sigma(X,Y),U\rangle=0,
$$
$$
\langle\nabla_X^{\perp}N,U\rangle=\langle\bar\nabla_X(JH-T),U\rangle=\langle-JA_HX-\sigma(X,T),U\rangle=0,
$$
and
$$
\langle\nabla_X^{\perp}JN,U\rangle=\langle\bar\nabla_X(-H-JT),U\rangle=\langle -J\sigma(X,T),U\rangle=0,
$$
that show that $L$ is parallel. We again use \cite[Theorem~2]{ET} to conclude that $\Sigma^2$ lies in $\mathbb{C}P^3(\rho)$. 

In the following we shall determine the shape operators $A_3$, $A_4$, $A_5$, and $A_6$. First, since $N$ is orthogonal to $H$ and to $J(T\Sigma^2)$, we have $A_6=0$. 

Next, since $\Sigma^2$ is totally real, we have
$$
\langle\sigma(X,Y),JZ\rangle=\langle\sigma(X,Z),JY\rangle,\quad\forall X,Y,Z\in C(T\Sigma^2),
$$
and, using this property, together with 
$$
\trace A_3=2\langle E_3,H\rangle=-2|T|\quad\textnormal{and}\quad\trace A_4=2\langle E_4,H\rangle=0,
$$
we see that $A_3$ and $A_4$ can be written as
$$
A_3=\left(\begin{array}{cc}a-|T|&b\\b&-a-|T|\end{array}\right)\quad\textnormal{and}\quad A_4=\left(\begin{array}{cc}b&-a-|T|\\-a-|T|&-b\end{array}\right).
$$
As for $A_5$, we have $\trace A_5=2\langle E_5,H\rangle=-2|N|$ and then
$$
A_5=\left(\begin{array}{cc}c-|N|&d\\d&-c-|T|\end{array}\right).
$$

Taking into account that $E_5$ is orthogonal to $J(T\Sigma^2)$ we obtain $\langle R(X,Y)H,E_5\rangle=0$ and, from the Ricci equation \eqref{Ricci}, one sees that $[A_H,A_5]=0$ and then, from \eqref{eq:H}, that $[A_3,A_5]=0$. After a straightforward computation, we get
\begin{equation}\label{eq:1}
ad=bc.
\end{equation}

Next, we have $\langle R(E_2,E_1)H,JE_2\rangle=-(\rho/4)|T|$ and then, again using the Ricci formula \eqref{Ricci}, 
$$
\langle[A_H,A_4]E_1,E_2\rangle=\langle R(E_2,E_1)H,JE_2\rangle=-\frac{\rho}{4}|T|,
$$
which can be written as
\begin{equation}\label{eq:2}
b(b|T|+d|N|)+(a+|T|)(a|T|+c|N|)=\frac{\rho}{8}|T|.
\end{equation}

Since $\Sigma^2$ is proper-biharmonic, from the first equation of \eqref{eq:bih}, tacking into account that $E_4$ is orthogonal to $H$ and to $E_3$, we see that
$\trace(A_HA_4)=0$, which, using \eqref{eq:1}, gives
\begin{equation}\label{eq:3}
b|T|+d|N|=0.
\end{equation}
Assume now that there exists a point $p\in\Sigma^2$ such that $b\neq 0$ or $d\neq 0$ at $p$. Then, from \eqref{eq:1}, \eqref{eq:2}, and \eqref{eq:3}, we obtain that $|T|=0$ at $p$, which is a contradiction. Therefore, from \eqref{eq:3}, it follows that $b=d=0$ on $\Sigma^2$, and then equation \eqref{eq:2} becomes
\begin{equation}\label{eq:4}
(a+|T|)(a|T|+c|N|)=\frac{\rho}{8}|T|.
\end{equation}

Finally, again using the first equation of \eqref{eq:bih}, we have 
$$
\trace(A_HA_3)=-\frac{5\rho}{4}|T|\quad\textnormal{and}\quad\trace(A_HA_5)=-\frac{\rho}{2}|N|,
$$
or, equivalently,
\begin{equation}\label{eq:5}
a(a|T|+c|N|)=\frac{(5\rho-8|H|^2)|T|}{8}
\end{equation}
and
\begin{equation}\label{eq:6}
c(a|T|+c|N|)=\frac{(\rho-4|H|^2)|N|}{4},
\end{equation}
respectively. From \eqref{eq:4}, \eqref{eq:5}, and \eqref{eq:6} one obtains 
\begin{equation}\label{eq:7}
a=\frac{(5\rho-8|H|^2)|T|}{4(2|H|^2-\rho)},\quad c=\frac{(\rho-4|H|^2)|N|}{2(2|H|^2-\rho)}
\end{equation}
and
\begin{equation}\label{eq:8}
16|H|^4-10\rho|H|^2-3\rho|T|^2+2\rho^2=0.
\end{equation}

The surface $\Sigma^2$ is flat and, therefore, from its Gauss equation \eqref{Gauss} in $\mathbb{C}P^n(\rho)$, follows
$$
0=K=\frac{\rho}{4}+\sum_{\alpha=3}^{5}\det A_{\alpha},
$$
which, together with \eqref{eq:7}, gives
\begin{equation}\label{eq:9}
16|H|^4+4\rho|H|^2-48|T|^2|H|^2+22\rho|T|^2-4\rho^2=0.
\end{equation}
From \eqref{eq:8} and \eqref{eq:9} we obtain $|H|^2=\rho/3$, $|T|^2=4\rho/27$, and $|N|^2=5\rho/27$. Hence, the shape operator $A$ is given by
\begin{equation}\label{eq:10}
A_3=\frac{1}{2}\sqrt{\frac{\rho}{3}}\left(\begin{array}{cc}-\frac{11}{3}&0\\0&1\end{array}\right),\quad A_4=\frac{1}{2}\sqrt{\frac{\rho}{3}}\left(\begin{array}{cc}0&1\\1&0\end{array}\right),\quad A_5=-\frac{1}{2}\sqrt{\frac{5\rho}{3}}\left(\begin{array}{cc}-\frac{1}{3}&0\\0&1\end{array}\right).
\end{equation}

Now, since $E_1$ and $E_2$ are parallel, they determine two distributions which are mutually orthogonal, smooth, involutive and parallel. Therefore, from the de Rham Decomposition Theorem follows, also tacking into account that the surface is complete and using its universal cover if necessary, that $\Sigma^2$ is the standard product $\gamma_1\times\gamma_2$, where $\gamma_k:\mathbb{R}\rightarrow\mathbb{C}P^3(\rho)$, $k\in\{1,2\}$, are integral curves of $E_1$ and $E_2$, respectively, parametrized by arc-length, i.e., $\gamma_1'=E_1$ and $\gamma_2'=E_2$ (see \cite{KN}). In the following, we shall determine these curves in terms of their curvatures and complex torsions.

Let us denote by $\kappa_i$, $1\leq i<6$, the curvatures of $\gamma_1$ and by $\{X_j^1\}$, $1\leq j<7$, its Frenet frame field. Using \eqref{eq:10}, we first have 
$$
\bar\nabla_{E_1}E_1=\sigma(E_1,E_1)=-\frac{11}{6}\sqrt{\frac{\rho}{3}}E_3-\frac{1}{6}\sqrt{\frac{5\rho}{3}}E_5
$$
and then, the first Frenet equation of $\gamma_1$ gives 
$$
\kappa_1=\sqrt{\frac{7\rho}{6}}\quad\textnormal{and}\quad X_2^1=-\frac{11\sqrt{14}}{42}E_3-\frac{\sqrt{70}}{42}E_5.
$$

Next, since $\nabla E_1=\nabla E_2=0$ and $0=\nabla^{\perp}H=-|T|\nabla^{\perp}E_3-|N|\nabla^{\perp}E_5$, we get
$$
\langle\nabla^{\perp}_{E_1}E_3,E_4\rangle=0\quad\textnormal{and}\quad\langle\nabla^{\perp}_{E_1}E_3,E_5\rangle=0
$$
and
\begin{align*}
\langle\nabla^{\perp}_{E_1}E_3,E_6\rangle &=\frac{1}{|N|}\langle\bar\nabla_{E_1}JE_1,JH-T\rangle=\frac{1}{|N|}(\langle A_HE_1,E_1\rangle+|T|\langle\bar\nabla_{E_1}E_1,E_3\rangle)\\ &=-\langle A_5E_1,E_1\rangle=\frac{1}{6}\sqrt{\frac{5\rho}{3}}.
\end{align*}
It follows that $\nabla^{\perp}_{E_1}E_3=(1/6)\sqrt{5\rho/3}E_6$. In the same way, one obtains $\nabla^{\perp}_{E_1}E_3=(-1/3)\sqrt{\rho/3}E_6$. Thus, after a straightforward computation, we have
$$
\bar\nabla_{E_1}X_2^1=-\kappa_1E_1-\frac{1}{2}\sqrt{\frac{5\rho}{42}}E_6,
$$
which means that 
$$
\kappa_2=\frac{1}{2}\sqrt{\frac{5\rho}{42}}\quad\textnormal{and}\quad X_3^1=-E_6.
$$
It follows that 
$$
\bar\nabla_{E_1}X_3^1=\frac{1}{6}\sqrt{\frac{5\rho}{3}}E_3-\frac{1}{3}\sqrt{\frac{\rho}{3}}E_5
$$
and then
$$
\kappa_3=\frac{3}{2}\sqrt{\frac{\rho}{42}}\quad\textnormal{and}\quad X_4^1=\frac{\sqrt{70}}{42}E_3-\frac{11\sqrt{14}}{42}E_5.
$$
Finally, we get $\bar\nabla_{E_1}X_4^1=-\kappa_3X_3^1$ and, therefore, $\gamma_1$ is a helix of osculating order $4$. A simple computation gives its complex torsions
$$
\tau_{12}=-\tau_{34}=\frac{11\sqrt{14}}{42},\quad \tau_{23}=-\tau_{14}=\frac{\sqrt{70}}{42},\quad \tau_{13}=\tau_{24}=0.
$$
Hence $\gamma_1$ is a holomorphic helix of order $4$. 

Consider now the subbundle $L=\Span\{E_3,E_5,E_6\}$ in the normal bundle of $\gamma_1$. It is easy to see that $L$ is parallel and $T\gamma_1\oplus L$ is invariant by $J$ and $\bar R$. Then, since $X_2^1\in L$, we apply \cite[Theorem 2]{ET} to conclude that $\gamma_1$ lies in $\mathbb{C}P^2(\rho)$. Moreover, it can be easily verified that $\gamma_1$ is of class $I_3$. 

For the curve $\gamma_2$ we have 
$$
\bar\nabla_{E_2}E_2=\sigma(E_2,E_2)=\frac{1}{2}\sqrt{\frac{\rho}{3}}E_3-\frac{1}{2}\sqrt{\frac{5\rho}{3}}E_5,
$$
and then its first curvature is $\kappa=\sqrt{\rho/2}$ and $X_2^2=(\sqrt{6}/6)E_3-(\sqrt{30}/6)E_5$. It can be easily verified that $\nabla^{\perp}_{E_2}E_3=\nabla^{\perp}_{E_2}E_5=0$ and then one obtains $\bar\nabla_{E_2}X_2^2=-\kappa E_2$. Therefore, the curve $\gamma_2$ is a holomorphic circle in $\mathbb{C}P^3(\rho)$ with curvature $\kappa=\sqrt{\rho/2}$ and complex torsion $\tau_{12}=0$. Then, we use Theorems \ref{MA1} and \ref{MA2} to conclude.  
\end{proof}

\begin{remark} Working in the same way as in the case when $\rho=4$, considered in \cite{S-GJM}, the result in Theorem \ref{main theorem}(2) can be extended to surfaces in $\mathbb{C}P^n(\rho)$. However, for the sake of simplicity we present here only this particular case.
\end{remark}


\begin{thebibliography}{99}

\bibitem{AdCT} H.~Alencar, M.~do Carmo, and R.~Tribuzy, \textit{A
Hopf Theorem for ambient spaces of dimensions higher than three},
J. Differential Geom. 84(2010), 1--17.

\bibitem{BMO1} A.~Balmu\c s, S.~Montaldo, and C.~Oniciuc, \textit{Classification results for biharmonic submanifolds in
spheres}, Israel J. Math. 168(2008), 201--220.

\bibitem{BMO-AM} A.~Balmu\c s, S.~Montaldo, and C.~Oniciuc, \textit{Biharmonic PNMC submanifolds in spheres}, Ark. Mat., to appear, arXiv:1110.4258.

\bibitem{BO} A.~Balmu\c s and C.~Oniciuc, \textit{Biharmonic submanifolds with parallel mean curvature vector field in spheres}, J. Math. Anal. Appl. 386(2012), 619--630.

\bibitem{B} M.~Batista, \textit{Simons type equation in $\mathbb{S}^2\times\mathbb{R}$ and
$\mathbb{H}^2\times\mathbb{R}$ and applications}, Ann. Inst.
Fourier (Grenoble) 61(2011), 1299--1322.

\bibitem{BYC} B. Y.~Chen, \textit{A report on submanifolds of finite type}, Soochow J.
Math. 22(1996), 117--337.

\bibitem{CO} B.-Y. Chen and K. Ogiue, {\it On totally real submanifolds}, Trans. Amer. Math. Soc. 193(1974), 257--266.

\bibitem{CY} S.-Y. Cheng and S.-T. Yau, \textit{Hypersurfaces with constant scalar curvature}, Math. Ann. 225(1977), 195--204.

\bibitem{JEJS} J.~Eells and J. H.~Sampson, \textit{Harmonic
mappings of Riemannian manifolds}, Amer. J. Math. 86(1964),
109--160.

\bibitem{ET} J. H.~Eschenburg and R.~Tribuzy, \textit{Existence and
uniqueness of maps into affine homogeneous spaces}, Rend. Sem.
Mat. Univ. Padova 89(1993), 11--18.

\bibitem{DF} D. Fetcu, \textit{Surfaces with parallel mean curvature vector in complex space forms}, J. Differential Geom. 91(2012), 215--232.

\bibitem{FLMO} D.~Fetcu, E.~Loubeau, S.~Montaldo, and C.~Oniciuc,
\textit{Biharmonic submanifolds of $\mathbb{C}P^n$}, Math. Z. 266(2010), 505--531.

\bibitem{FO-TMJ} D. Fetcu and C. Oniciuc, {\it Biharmonic integral $\mathcal{C}$-parallel submanifolds in
$7$-dimensional Sasakian space forms}, Tohoku Math. J. 64(2012), 195--222.

\bibitem{FOR} D. Fetcu, C. Oniciuc, and H. Rosenberg, {\it Biharmonic submanifolds with parallel mean
curvature in $\mathbb{S}^n\times\mathbb{R}$}, J. Geom. Anal., to appear, arXiv:1109.6138.

\bibitem{H} A.~Huber, \textit{On subharmonic functions and differential geometry in the
large}, Comment. Math. Helv. 32(1957), 13--71.

\bibitem{GYJ} G. Y.~Jiang, \textit{$2$-harmonic maps and
their first and second variational formulas}, Chinese Ann. Math.
Ser. A7(4)(1986), 389--402.

\bibitem{KN} S. Kobayashi and K. Nomizu, {\it Foundations of Differential Geometry}, vol. I, Interscience Publishers, New York, London, 1963.

\bibitem{LO} E. Loubeau and C. Oniciuc, {\it Biharmonic CMC surfaces in spheres}, preprint 2013. 

\bibitem{MA} S.~Maeda and T.~Adachi, {\it Holomorphic helices in a complex space form}, Proc. Amer. Math. Soc. 125(1997), 1197--1202.

\bibitem{SMYO} S.~Maeda and Y.~Ohnita, {\it Helical geodesic immersions into complex space forms}, Geom. Dedicata 30(1989), 93--114.
 
\bibitem{U} S. Maeta and H. Urakawa, {\it Biharmonic Lagrangian submanifolds in K\"ahler manifolds}, Glasg. Math. J., to appear, arXiv:1203.4092.

\bibitem{OW} Y.-L. Ou and Z.-P. Wang, \textit{Constant mean curvature and totally umbilical biharmonic surfaces in $3$-dimensional geometries}, J. Geom. Phys. 61(2011), 1845--1853.

\bibitem{R} H. Reckziegel, {\it Horizontal lifts of isometric immersions into the bundle space of a pseudo-Riemannian submersion}, Global differential geometry and global analysis 1984 (Berlin, 1984), 264--279, Lecture Notes in Math., 1156, Springer, Berlin, 1985.

\bibitem{S-GJM} T. Sasahara, {\it Biharmonic Lagrangian surfaces of constant mean curvature in complex space forms}, Glasg. Math. J. 49(2007), 497--507.

\bibitem{Sato1} N.~Sato, \textit{Totally real submanifolds of a
complex space form with nonzero parallel mean curvature vector},
Yokohama Math. J. 44(1997), 1--4.

\bibitem{JS} J.~Simons, \textit{Minimal varieties in Riemannian
manifolds}, Ann. of Math. 88(1968), 62--105.

\bibitem{Z} W. Zhang, {\it New examples of biharmonic submanifolds in $\mathbb{C}P^n$ and $\mathbb{S}^{2n+1}$}, An. \c Stiin\c t. Univ. Al. I. Cuza Ia\c si. Mat. (N.S.) 57(2011), 207--218.

\end{thebibliography}
\end{document}